\documentclass[11pt]{amsart}

\usepackage{amsmath,amssymb,amsthm}
\usepackage{amsfonts}
\usepackage[mathscr]{eucal}
\newtheorem{theorem}{Theorem}[section]
\newtheorem{lemma}[theorem]{Lemma}
\newtheorem{prop}[theorem]{Proposition}

\theoremstyle{definition}

\theoremstyle{remark}
\newtheorem{remark}[theorem]{Remark}

\numberwithin{equation}{section}

\newcommand{\RR}[1]{\mathbb{#1}}

\newcommand{\rr}{{\mathbb R}}

\def\R{{\mathbb R}}
\def\N{{\mathbb N}}

\def\C{{\mathbb C}}

\def\a{\alpha}

\def\E{{\mathbb E}}
\def\P{{\mathbb P}}

\begin{document}

\title{\bf Large deviations for local time fractional Brownian
motion and applications}

\author{Mark M. Meerschaert}
\address{Mark M. Meerschaert, Department of Probability and Statistics,
Michigan State University, East Lansing, MI 48823}
\email{mcubed@stt.msu.edu}
\urladdr{http://www.stt.msu.edu/$\sim$mcubed/}
\thanks{Research of M. M. Meerschart was partially
supported by NSF grant DMS-0706440.}

\author{Erkan Nane}
\address{Erkan Nane, Department Statistics and Probability,
Michigan State University, East Lansing, MI 48823}
\email{nane@stt.msu.edu}
\urladdr{http://www.stt.msu.edu/$\sim$nane}

\author{Yimin Xiao}
\address{Yimin Xiao, Department Statistics and Probability,
Michigan State University, East Lansing, MI 48823}
\email{xiao@stt.msu.edu}
\urladdr{http://www.stt.msu.edu/$\sim$xiaoyimi}
\thanks{Research of Y. Xiao was partially supported by
NSF grant DMS-0706728.}

\begin{abstract}
Let $W^H=\{W^H(t), t \in \rr\}$ be a fractional Brownian motion of Hurst
index $H \in (0, 1)$ with values in $\rr$, and let $L = \{L_t, t \ge 0\}$
be the local time process at zero of a strictly stable L\'evy process
$X=\{X_t, t \ge 0\}$ of index $1<\alpha\leq 2$ independent of $W^H$.
The $\a$-stable local time fractional Brownian motion $Z^H=\{Z^H(t),
t \ge 0\}$ is defined by $Z^H(t) = W^H(L_t)$. The process $Z^H$ is
self-similar with self-similarity index $H(1 - \frac 1 \alpha)$ and is
related to the scaling limit of a continuous time random walk with
heavy-tailed waiting times between jumps (\cite{coupleCTRW,limitCTRW}).
However, $Z^H$ does not have stationary increments and is non-Gaussian.

In this paper we establish large deviation results for the process $Z^H$.
As applications we derive upper bounds for the uniform modulus of continuity
and the laws of the iterated logarithm for $Z^H$.
\end{abstract}


\maketitle

\section{Introduction}

Self-similar processes arise naturally in limit theorems of random walks and
other stochastic processes, and they have been applied to model various phenomena
in a wide range of scientific areas including telecommunications, turbulence,
image processing and finance. The most important example of self-similar
processes is fractional Brownian motion (fBm) which is a centered Gaussian process
$W^H=\{W^H(t), t \in \rr\}$ with $W^H(0) = 0$ and covariance function
\begin{equation}\label{Eq:fbmcov}
\E\big(W^H(s) W^H(t)\big) = \frac 1 2 \Big(|s|^{2H} + |t|^{2H} - |s-t|^{2H}\Big),
\end{equation}
where $H \in (0, 1)$ is a constant. By using \eqref{Eq:fbmcov} one can verify that
$W^H$ is self-similar with index $H$ (i.e., for all constants $c > 0$, the
processes $\{W^H(ct), t \in \rr\}$ and $\{c^H W^H(t), t \in \rr\}$ have the same
finite-dimensional distributions) and has stationary increments. When $H=1/2$,
$W^H$ is a two-sided Brownian motion, which will be written as $W$.

Many authors have constructed and investigated various classes of non-Gaussian
self-similar processes. See, for example, \cite{ST94} for information on self-similar
stable processes with stationary increments. Burdzy \cite{burdzy1, burdzy2}
introduced the so-called \emph{iterated Brownian motion} (IBM) by replacing the time
parameter in $W$ by an independent one-dimensional Brownian motion $B= \{B_t,
t \ge 0\}$. His work inspired  many researchers to explore the
connections between IBM (or other iterated processes) and PDEs
\cite{allouba1, allouba2, bmn-07,nane3}, to establish potential theoretical
results \cite{bandeb,deblassie,nane,nane2,nane5,nane6}
and to study its sample path properties \cite{burdzy-khos,CF,CCFR3,CCFR,hu-2,
hu,koslew, xiao}.

In this paper, we consider another class of iterated self-similar
processes which is related to continuous-time random walks considered in
\cite{coupleCTRW,limitCTRW}. Let $W^H=\{W^H(t), t \in \rr\}$ be a fractional
Brownian motion of Hurst index $H \in (0, 1)$ with values in $\rr$. Let $X =
\{X_t, t \ge 0\}$ be a real-valued, strictly stable L\'evy process of index
$1<\alpha\leq 2$. We assume that $X$ is independent of $W^H$. Let
$L= \{L_t, t \ge0\}$ be the local time process at zero of $X$ (see
Section 2 for its definition). Let $Z^H= \{Z^H(t), t \ge 0\}$ be a real-valued
stochastic process defined by $Z^H(t) = W^H(L_t)$ for all $t \ge 0$. This
iterated process will be called an \emph{$\a$-stable local time $H$-fractional
Brownian motion} or simply \emph{local time fractional Brownian motion}.

Since the sample functions of $W^H$ and $L$ are a.s. continuous, the local
time fractional Brownian motion $Z^H$ also has continuous sample functions.
Moreover, by using the facts that $W^H$ is self-similar with index $H$ and
$L$ is self-similar with index $1-1/\alpha$, one can readily verify that
$Z^H$ is self-similar with index $H(1 - 1/\alpha)$. However, $Z^H$ is
non-Gaussian, non-Markovian and does not have stationary increments.
When $H=1/2$, we will call $Z^{1/2}$ the local time Brownian motion and
denote it by $Z$ for convenience.

The local time Brownian motion $Z$ emerges as the scaling
limit of a continuous time random walk with heavy-tailed waiting times
between jumps \cite{coupleCTRW, limitCTRW}. Moreover, local time Brownian
motion has a close connection to fractional partial
differential equations. Baeumer and Meerschaert \cite{fracCauchy} showed
that the process $Z$ can be
applied to provide a solution to the fractional Cauchy problem. More precisely,
they proved that, if $L_t$ is the local time at 0 of a symmetric stable
L\'evy process, then $u(t,x)=\E_x [f(W(L_t))]$ solves the following fractional
in time PDE
\begin{equation}\label{frac-derivative-0}
\frac{\partial^\beta}{\partial t^\beta}u(t,x) =
{\Delta_x}u(t,x); \quad u(0,x) = f(x),
\end{equation}
where $\beta =1-1/\alpha$ and  $\partial^{\beta} g(t)/\partial t^\beta $ is
the Caputo fractional derivative in time, which can be defined as the inverse
Laplace transform of $s^{\beta}\tilde{g}(s)-s^{\beta -1}g(0)$, where
$\tilde{g}(s)=\int_{0}^{\infty}e^{-st}g(t)dt$ is the usual Laplace
transform. Recently Baeumer, Meerschaert and Nane \cite{bmn-07} further
established the
equivalence of the governing PDEs of $W(L_t)$ and $W(|B_{t}|)$ when
$\alpha=2$ and $\beta=1/2$. Here $B=\{B_t, t \ge 0\}$ is another Brownian
motion independent of $W$ and $X$. The process $Z$ has also appeared in
the works of Borodin \cite{borodin1, borodin2}, Ikeda and Watanabe
\cite{ike-wata}, Kasahara \cite{kasahara}, and Papanicolaou et al.
\cite{papanicolaou}. In \cite{CCFR2}, Cs\'{a}ki,  F\"{o}ldes and
R\'{e}v\'{e}sz  studied the Strassen type law of the iterated logarithm
of $Z(t) = W(L_t)$ when $L_t$ is the local time at zero of a symmetric
stable L\'{e}vy process.

For all $H \in (0, 1)$ and $\alpha \in (1, 2]$, $\a$-stable local time
$H$-fractional Brownian motions form a new class of self-similar processes.
It is natural to expect that they arise as scaling limit of continuous-time
\emph{correlated} random walks with heavy-tailed waiting times and, as such,
they are potentially useful as stochastic models. Hence it is of interest in
both theory and applications to investigate their probabilistic
and analytic properties. Due to the non-Gaussian and non-Markovian nature
of local time fractional Brownian motions, the existing theories on Markov
and/or Gaussian processes can not be applied to them directly and some new
tools will have to be developed. The literature on iterated Brownian motion
mentioned above provides an instructive guideline for studying local time
fractional Brownian motions.

The objective of the present paper is to establish large deviation results
for the local time fractional Brownian motion $Z^H$ and apply them to study
regularity properties of the sample paths of $Z^H$. We will consider the
interesting problem of determining the domain of attraction of $Z^H$ in a
subsequent paper.

\medskip

The following Theorems \ref{Thm:Zt-LDP} and \ref{probability-thm} are our
main results. 

\begin{theorem}\label{Thm:Zt-LDP}
Let $Z^H= \{Z^H(t), t \ge 0\}$ be an $\a$-stable local time $H$-fractional
Brownian motion with values in $\rr$ and $2H < \alpha$. Then for every Borel set
$D\subseteq \rr$,
\begin{equation}\label{Eq:LDPup}
\limsup_{t \to \infty} t^{- \frac{2H(\alpha - 1)} {\alpha - 2H}}\, \log \P\Big\{
t^{- \frac{2H(\alpha - 1)} {\alpha - 2H}} Z^H(t) \in D\Big\}
\le - \inf_{x \in \overline{D}}\Lambda^*_1(x)
\end{equation}
and
\begin{equation}\label{Eq:LDPlow}
\liminf_{t \to \infty} t^{- \frac{2H(\alpha - 1)} {\alpha - 2H}}\, \log \P\Big\{
t^{- \frac{2H(\alpha - 1)} {\alpha - 2H}} Z^H(t) \in D\Big\}
\ge - \inf_{x \in {D}^{\circ}}\Lambda^*_1(x),
\end{equation}
where $\overline{D}$ and $D^{\circ} $ denote respectively the closure and
interior of $D$ and
\begin{equation}\label{Eq:Lambdastar}
\Lambda^*_1(x) =
\frac{\alpha + 2H} {2 \alpha} \left(\frac{\alpha - 2H}
{2 \alpha B_1}\right)^{\frac{\alpha - 2H} {\alpha + 2H}}\, x^{\frac{2\alpha}
{\alpha + 2H}}, \qquad  \forall\ x \in \R.
\end{equation}
In the above,  $B_1 = B_1(H, \alpha, \chi, \nu)$ is the positive constant defined by
\begin{equation}\label{Eq:ConstB}
B_1= \frac{\alpha - 2H} {2\alpha} \left( \frac{H\, A_1^{\alpha} }
{(1 - \frac 1 \alpha)^{\alpha -1} }\right)^{\frac{2H} {\alpha - 2H}},
\end{equation}
where $A_1$ is the constant given by
\begin{equation}\label{Eq:A1}
A_1 = \frac{\Gamma(1-\frac 1 \alpha)\Gamma(\frac 1 \alpha)\chi^{1/\alpha}
\cos (\frac 1 \alpha \arctan
(\nu \tan (\frac{\pi \alpha}{2})))}{\pi \alpha [1+(\nu\tan
(\frac{\pi \alpha}{2}))^2]^{1/(2\alpha)}}
\end{equation}
and  $\nu \in [-1,1]$ and $ \chi>0$ are the parameters of the stable L\'evy
process $X$ defined in (\ref{Eq:Chf}).
\end{theorem}

In (\ref{Eq:Lambdastar}) and the sequel,  for any $\gamma >0$ and $x\in \R$,
the term $x^{2\gamma}$ is defined as $(x^2)^{\gamma}$. Since $2H < \alpha$,
one can see that the function $\Lambda^*_1$ in  (\ref{Eq:Lambdastar}) is even,
convex and differentiable on $\R$.

In the terminology of \cite{DemboZ98}, Theorem \ref{Thm:Zt-LDP} states that the
pair $\big(t^{- \frac{2H(\alpha - 1)} {\alpha - 2H}} Z^H(t),
t^{\frac{2H(\alpha - 1)} {\alpha - 2H}}\big)$ satisfies a large deviation principle
with good rate function $\Lambda^*_1$.  When $H=1/2$,
it yields a large deviation result for the local time Brownian motion and, moreover,
the constants $B_1$ in (\ref{Eq:ConstB}) can be simplified.

Letting $D = [x, \infty)$, we derive from Theorem \ref{Thm:Zt-LDP} and the
self-similarity of $Z^H$ the asymptotic tail probability $\P\big\{Z^H(1) \ge x\big\}$
as $x \to \infty$. The following theorem is more general because it holds for all
$H \in (0, 1)$ and $\alpha \in (1, 2]$.

\begin{theorem}\label{probability-thm}
Let $Z^H= \{Z^H(t), t \ge 0\}$ be an $\a$-stable local time $H$-fractional
Brownian motion with values in $\rr$. Then for any $0\leq a\leq b<\infty$,
\begin{equation}\label{Eq:LD}
\lim_{x \to\infty} \frac{\log \P \left\{ \big|Z^H(b)-Z^H(a)\big|>x
\right\}}{x^{\frac{2\alpha}{\alpha + 2H}}} = - B_2,
\end{equation}
where $B_2 = B_2(H, \alpha, \chi, \nu)$ is the positive constant defined by
\begin{equation}\label{Eq:conB2}
B_2 = \frac{\alpha + 2H} {2 \alpha}\, \left(\frac{H\, {A}_1^\alpha}
{\big(1 - \frac 1 \alpha\big)^{\alpha - 1}}\,
\right)^{- \frac{2 H} {\alpha + 2H}}\, \big(b-a\big)^{- \frac{2H(\alpha - 1)}
{\alpha + 2H}}.
\end{equation}
\end{theorem}


In order to prove Theorems \ref{Thm:Zt-LDP} and \ref{probability-thm}, we first
study the analytic properties of the moment generating functions of $Z^H(t)$ and
$|Z^H(b)-Z^H(a)|$. This is done by calculating the moments of $Z^H(t)$ and
$|Z^H(b)-Z^H(a)|$ for $0\leq a\leq b $ directly and by using a theorem of
Valiron \cite{valiron}. Then Theorems \ref{Thm:Zt-LDP} and \ref{probability-thm}
follow respectively from the G\"artner-Ellis Theorem
(cf. \cite{DemboZ98}) and a result of Davies \cite{davies}.

The rest of the paper is organized as follows. In $\S$2, we derive
sharp estimates on the moments of the local time $L_t$ of $X$, and the moments of
$Z^H(b)-Z^H(a)$. These estimates are applied in $\S$3 to study the analyticity
of the moment generating functions of $Z^H$, and to derive large time behavior of
the logarithmic moment generating functions $\log \E\big[\exp(\theta Z^H(t))\big]$ and
$\log \E\big[\exp(t|Z^H(b)-Z^H(a)|^\beta)\big]$ for suitably chosen $\beta> 0$.
In $\S$ 4, we prove Theorem \ref{Thm:Zt-LDP} and \ref{probability-thm}.
We will also establish similar tail estimates for the maximum
$\max_{t \in [a, b]}|Z^H(t)-Z^H(a)|$. In $\S$5, by combining the large deviation
result with the methods in \cite[Theorem 3.1]{csaki-csorgo}, we establish
local and uniform moduli of continuity for $Z^H$. We also obtain an upper bound in
the law of the iterated logarithm for $Z^H$.

\bigskip

\noindent{\bf Acknowledgment}\ This paper is finished while Y. Xiao is
visiting the Statistical \& Applied Mathematical Sciences Institute
(SAMSI). He thanks the staff of SAMSI for their
support and the good working conditions.

\section{Moment estimates}

A L\'{e}vy process $X=\{X_t, t \ge 0\}$ with values in $\rr$ is called strictly
stable of index $\alpha \in (0, 2]$ if its characteristic function is given by
\begin{equation}\label{Eq:Chf}
\E\big[\exp (i\xi X_t)\big] = \exp\left(-t |\xi|^{\alpha} \frac{1+i \nu {\rm sgn}
(\xi)\tan (\frac{\pi\alpha}{2})}{\chi}\right),
\end{equation}
where $-1\leq  \nu \leq 1$ and $\chi >0$ are constants. In the terminology of
\cite[Definition 1.1.6]{ST94},
$\nu$ and $\chi^{-1/\a}$ are respectively the skewness and scale parameters
of the stable random variable $X_1$. When $\a=2$ and $\chi = 2$, $X$ is Brownian
motion. In general, many properties of stable L\'evy processes can be
characterized by the parameters $\alpha, \nu$ and $\chi$. For a systematic account
on L\'evy processes we refer to \cite{Bertoin96}.

For any Borel set $I\subseteq \RR{R}$, the occupation measure of $X$ on $I$
is defined by
\begin{equation}\label{occup-measure}
\mu_{I}(A)=\lambda_{1}\{t\in I: \ \ X_t \in A\}
\end{equation}
for all Borel sets $A\subseteq \rr$, where $\lambda_{1}$ is the one-dimensional
Lebesgue measure. If $\mu_{I}$ is absolutely
continuous with respect to the Lebesgue measure $\lambda_{1}$ on
$\RR{R}$, we say that $X$ has a local time on $I$ and
define its local time $L(x,I)$ to be the Radon-Nikod\'ym derivative
of $\mu_{I}$ with respect to $\lambda_1$, i.e.,
\[
    L(x,I) = \frac{d\mu_{I}} {d\lambda_1}(x),\qquad \forall x\in\rr.
\]
In the above, $x$ is the so-called \emph{space variable}, and $I$
is the \emph{time variable} of the local time. If $I=[0,t]$, we will
write $L(x,I)$ as $L(x,t)$. Moreover, if $x=0$ then we will simply write
$L(0,t)$ as $L_t$.

By using a monotone class argument, one can verify that $L(x, I)$
satisfies the following \emph{occupation
density formula}: For every measurable function $f : \R \to \R_+$,
\begin{equation}\label{Eq:occupation}
\int_I f(X(t))\, d t = \int_{\R} f(x) L(x, I)\, dx.
\end{equation}

It is well-known (see, e.g. \cite{Bertoin96}) that a strictly stable L\'evy process
$X$ has a local time if and only if $\a \in (1, 2]$. In the later case, $L(x, t)$
has a version that is continuous in $(x, t)$. Throughout this paper, we tacitly
work with such a version so that the local time process $L=\{L_t, t \ge 0\}$
has continuous sample paths.

It follows from (\ref{Eq:occupation}) and the
self-similarity of $X$ that $L(x, t)$ has the following scaling property: For every
constant $c > 0$, $c^{1 - 1/\alpha}L(c^{-1/\alpha}x, \, c^{-1}t)$ is a version of
$L(x, t)$. In particular, by letting $x=0$ we see that $L_t$ is self-similar with
index $1 - \frac 1 \alpha$.

For the purpose of the present paper, it will be convenient to
express the local time $L(x,t)$ as the inverse Fourier transform
of $\widehat{\mu}(u, t):= \widehat{\mu_{[0, t]}}(u)$, namely
\begin{equation}\label{local-time-def}
\begin{split}
L(x,t)&= \frac{1}{2\pi} \int_{\RR{R}}\exp(-iux)\widehat{\mu}(u,t)\, du\\
&= \frac{1}{2\pi} \int_{0}^{t}\int_{\RR{R}}
\exp\big(-iux + iuX(s)\big)\, duds.
\end{split}
\end{equation}
This formal expression can be justified rigorously (see \cite{gem-hor}).
Moreover, it follows from (25.2) and (25.7) in \cite{gem-hor} that
for all $x\in \RR{R}$, $I\in \mathcal{B}(\RR{R}_{+})$ and all integers $n\geq
1$, we have
\begin{equation}\label{local-time-moment}
\E\big[L(x,I)\big]^{n}=
(2\pi)^{-n}\int_{I^{n}}\int_{\RR{R}^{n}}\exp\bigg(-i\sum_{j=1}^{n}u_{j}x\bigg)
\E\exp \bigg(i\sum_{j=1}^{n} u_{j}X(t_{j})\bigg)\, d\bar{u}d\bar{t}.
\end{equation}
In the above and in the sequel, $d\bar{u}=du_1 \cdots du_n$ and $d\bar{t}=
dt_1 \cdots dt_n$.

We start with the following moment estimates for the local time of $X$.

\begin{lemma}\label{Lem:LevyLT}
Let $L = \{L_t, t \ge 0\}$ be the local time at zero of a strictly stable
L\'evy  process $X= \{X_t, t \ge 0\}$ with values in $\rr$ and index
$1<\alpha\leq 2$. Then for all $0< a\leq b<\infty$ and all integers $n\geq 1$,
\begin{equation}\label{Eq:Ltmoments}
\begin{split}
\bigg(\frac{b-a}{b}\bigg)^{1/\alpha}\, \frac{A_1^n\, n!\, (b-a)^{n(1-1/\alpha)}}
{\Gamma\big(1 - \frac 1 \alpha\big) \Gamma\big(1 + \frac 1 \alpha +
n(1-\frac 1 \alpha)\big)}
&\leq \E\big[ \big| L_b-L_a\big|^{n} \big]\\
&  \leq \bigg(\frac{b-a}{b}\bigg)^{1/\alpha}\, \frac{A_1^n\, n!\,
(b-a)^{n(1-1/\alpha)}}
{\Gamma\big(1 + n(1-\frac 1 \alpha)\big)},
\end{split}
\end{equation}
where $A_1 >0$ is the constant defined by (\ref{Eq:A1}).
In the case $a=0$, we have the equality
\begin{equation}\label{Eq:a-zero}
\E\big[ |L_b|^{n} \big] =  \frac{A_1^{n} \,n!}
{\Gamma\big(1+n(1-\frac 1 \alpha)\big)}\, b^{n(1-1/\alpha)}.
\end{equation}
\end{lemma}

\begin{proof}\
Applying (\ref{local-time-moment}) with $x=0$ and $I = (a, b]$
and making a change of variables $s_j=t_j -a$ ($j=1, \cdots, n$), we obtain
\begin{equation}\label{stationary0}
\begin{split}
\E\big[|L_b - L_b|^n\big]&=
(2\pi)^{-n}\int_{(a,b]^{n}}\int_{\RR{R}^{n}} \E\exp \left(
i\sum_{j=1}^{n}u_{j}X(t_{j})\right)d\bar{u}d\bar{t}\\
&= \frac{n!} {(2\pi)^{n}}\int_{S_{n}}\int_{\RR{R}^{n}} \E\exp\left(
i\sum_{j=1}^{n}u_{j}X(s_{j}+a)\right)\,d\bar{u}d\bar{s},
\end{split}
\end{equation}
where
$$
S_n=\big\{(s_{1},\cdots , s_{n}):\ 0\leq s_{1}\leq s_{2}\leq \cdots
\leq s_{n} \leq b-a \big\}.
$$

Let $v_j=\sum_{l=j}^{n}u_l$ ($j=1, \cdots, n$), then the last sum in
(\ref{stationary0}) can be written as
\begin{equation}\label{stationary}
\sum_{j=1}^{n}u_{j}X(s_{j}+a) = \sum_{j=1}^{n}v_{j}\big(X(s_{j}+a)-X(s_{j-1}+a) \big),
\end{equation}
where $s_0 :=-a$ so that $s_0 + a=0$. Denote
$$
\phi(\xi)=|\xi|^{\alpha} \frac{1+i \nu\, {\rm sgn} (\xi)\tan
(\frac{\pi\alpha} 2)}{\chi}.
$$
Since the process $X$ has stationary and independent
increments,  we have
\begin{equation}\label{Eq:S2}
\begin{split}
&\E\exp \left(i\sum_{j=1}^{n}v_{j}\big(X(s_{j}+a)-X(s_{j-1}+a)\big)\right) \\
&= \exp\left(-\sum_{j=2}^{n}(s_j-s_{j-1})\phi(v_{j})\right)
\exp\big(-(s_1 +a)\phi (v_1)\big).
\end{split}
\end{equation}
It follows from (\ref{stationary0}), (\ref{stationary}), (\ref{Eq:S2})
and a change of variables that
\begin{equation}\label{Eq:S3}
\begin{split}
\E\big[|L_b - L_a|^n\big] &= \frac{n!}{(2\pi)^{n}}\int_{S_n}\int_{\RR{R}^{n}}
\exp\left(-\sum_{j=2}^{n}(s_j-s_{j-1}) \phi(v_{j})\right)\\
 &\qquad \qquad \qquad \times \exp\big(-(s_1 +a)\phi (v_1)\big)\,
 d\bar{v}d\bar{s},
 \end{split}
 \end{equation}
where $d\bar{v} = dv_1\cdots dv_n$.

We know from symmetry that for any $s>0$:
$$
\int_{\RR{R}}\exp\big(-s\, \phi(v)\big)\,dv =
2 \int_{0}^{\infty}\exp\bigg(- \frac{s\, v^{\alpha}} \chi\bigg)
\cos \bigg(- \frac{s v^{\alpha} \nu} \chi\, \tan \big( \frac{\alpha\pi}
2 \big)\bigg)\, dv.
$$
By a change of variable $z=v^{\alpha}$, this equals
$$
\frac{2}{\alpha}\int_{0}^{\infty}z^{1/\alpha -1}
\exp \bigg(- \frac{s\,z} \chi \bigg)
\cos \bigg(-\frac{s\,z \nu} \chi\, \tan \big(\frac{\alpha\pi} 2 \big)\bigg)\,dz,
$$
which by Equation 3.944(6) in Gradsteyn and Ryzhik \cite{grad-ryzhik} equals
$$
\frac{2 \chi^{1/\alpha} \Gamma (\frac 1 \alpha) }{\alpha\, s^{1/\alpha}
\big(1+[\nu\tan(\frac{\alpha \pi} 2)]^2\big)^{1/(2\alpha)}}\cos \left(\frac{1}{\alpha}
\arctan \big(\nu \tan(\frac{\alpha\pi} 2)\big)\right).
$$

Combining the above with (\ref{Eq:S3}) we obtain
\begin{equation}\label{time-integral}
\E\big[|L_b - L_a|^n\big]= n!\, C(\alpha)^n\int_{S_n}(s_1+a)^{-1/\alpha}\prod_{j=2}^{n}
(s_j-s_{j-1})^{-1/\alpha}\, d\bar{s},
\end{equation}
where $C(\alpha)$ is the constant given by
\begin{equation}\label{Eq:C}
C(\alpha)=\frac{\chi^{1/\alpha} \Gamma (\frac 1 \alpha)}{\pi \alpha (1+
[\nu\tan(\frac{\alpha \pi} 2)]^2)^{1/(2\alpha)}}\cos \left(\frac{1}{\alpha}
\arctan \big(\nu \tan(\frac{\alpha\pi} 2)\big)\right).
\end{equation}

We denote the multiple integral in (\ref{time-integral}) by $J_n$. When
$a=0$, it can be evaluated in terms of the Gamma function. When $a>0$
the same induction method can still be applied. We include a proof for
completeness.

First we integrate over $s_n \in [s_{n-1},\, b-a]$ to get
$$
\int_{s_{n-1}}^{b-a}(s_n-s_{n-1})^{-1/\alpha}ds_{n}=
\frac{(b-a-s_{n-1})^{1-1/\alpha}}{1- \frac 1 \alpha}.
$$
Next we integrate over $s_{n-1}\in [s_{n-2},\, b-a]$. By changing
variables twice $v=s_{n-1}-s_{n-2}$ and $t=v/(b-a-s_{n-2})$, we obtain
\begin{eqnarray}
& &\int_{s_{n-2}}^{b-a}(s_{n-1}-s_{n-2})^{-1/\alpha}
  \frac{(b-a-s_{n-1})^{1-1/\alpha}}{ 1- \frac 1 \alpha }\,
ds_{n-1}\nonumber\\
& & =\frac{1}{1- \frac 1 \alpha}
\int_{0}^{b-a-s_{n-2}}v^{-1/\alpha}(b-a-s_{n-2}-v)^{1-1/\alpha}dv\nonumber\\
& &=\frac{(b-a-s_{n-2})^{2(1-1/\alpha)}}{1- \frac 1 \alpha}
\int_{0}^{1}t^{(1-1/\alpha)-1}(1-t)^{(2-1/\alpha)-1}\, dt\nonumber\\
& &
=\frac{(b-a-s_{n-2})^{2(1-1/\alpha)}}{1- \frac 1 \alpha} \cdot\frac{\Gamma
(1- \frac 1 \alpha) \Gamma(2- \frac 1 \alpha)}{\Gamma (1+2(1- \frac 1 \alpha))
}\nonumber\\
&& =\frac{\Gamma
(1- \frac 1 \alpha)^2}{\Gamma
(1+2(1- \frac 1 \alpha)) }\, (b-a-s_{n-2})^{2(1-1/\alpha)}.\nonumber
\end{eqnarray}
Iterating this procedure, we derive
\begin{equation}\label{general-integral}
\begin{split}
J_{n} &= \frac{\Gamma (1- \frac 1 \alpha)^{n-1}}{\Gamma
(1+(n-1)(1- \frac 1 \alpha))}\int_{0}^{b-a} (b-a
-s_{1})^{(n-1)(1-1/\alpha)}(s_1+a)^{-1/\alpha}\, ds_1\\
&= \frac{\Gamma (1- \frac 1 \alpha)^{n-1}} {\Gamma
(1+(n-1)(1- \frac 1 \alpha))} \, b^{n(1-1/\alpha)} \int_{0}^{(b-a)/b}
 v^{(n-1)(1-1/\alpha)}(1-v)^{-1/\alpha}\, dv,
\end{split}
\end{equation}
where the second equality follows from change of variables.



If $a=0$, then the last integral equals
$\mathrm{Beta}\big(1- \frac 1 \alpha,\,1+(n-1)(1- \frac 1 \alpha) \big)$,
where $\mathrm{Beta}$ denotes the
Beta function.  This and (\ref{general-integral}) yield
\begin{equation}\label{Eq:Jn}
J_n=  \frac{\Gamma \big(1- \frac 1 \alpha \big)^n}
{\Gamma \big(1+n(1- \frac 1 \alpha)\big)}\, b^{n(1-1/\alpha)}.
\end{equation}

It follows from (\ref{time-integral}) and (\ref{Eq:Jn}) that for $a=0$,
\begin{equation}\label{Eq:LTa0}
\E\big[ |L_b - L_a|^n\big]=   \frac{\Gamma \big(1- \frac 1 \alpha \big)^{n}\,
C(\alpha)^n\, n!}{\Gamma \big(1+ n(1- \frac 1 \alpha) \big)}\, b^{n(1-1/\alpha)}.
\end{equation}
By (\ref{Eq:A1}) and (\ref{Eq:C}), we have $A_1 = \Gamma
\big(1- \frac 1 \alpha \big)\, C(\alpha)$. Hence the desired result (\ref{Eq:a-zero})
follows from (\ref{Eq:LTa0}).

If $a>0$, then last integral in (\ref{general-integral}) is an
incomplete Beta function \cite[8.384(1) and 9.100]{grad-ryzhik}. Write
\begin{equation}\label{Eq:f-f}
\begin{split}
f(n,a,b) &= \int_{0}^{(b-a)/b}  v^{(n-1)(1-1/\alpha)}(1-v)^{-1/\alpha}\, dv\\
&= \bigg(\frac {b- a} b \bigg)^{n (1 - 1/\alpha) + 1/\alpha}\,
\int_0^1 v^{(n-1)(1-1/\alpha)}\left(1- \frac{b-a} b v\right)^{-1/\alpha}\, dv.
\end{split}
\end{equation}
Then one can verify that
\begin{equation}\label{Eq:fupb}
\begin{split}
f(n,a,b) &\le \bigg( \frac {b-a}b\bigg)^{n (1 - 1/\alpha) + 1/\alpha} \int_0^{1}
v^{(n-1)(1-1/\alpha)} (1 - v)^{-1/\alpha}\, dv\\
&= \bigg(\frac {b- a} b \bigg)^{n (1 - 1/\alpha) + 1/\alpha}
\frac{\Gamma(1 - \frac 1 \alpha) \Gamma(1 + (n-1)(1- \frac 1 \alpha))}
{\Gamma(1 + n(1- \frac 1 \alpha))}
\end{split}
\end{equation}
and
\begin{equation}\label{Eq:flowb}
\begin{split}
f(n,a,b) &\ge  \bigg( \frac {b-a}b\bigg)^{n (1 - 1/\alpha)+ 1/\alpha } \int_0^{1}
v^{(n-1)(1-1/\alpha)} \, dv\\
&= \bigg(\frac {b- a} b \bigg)^{ n (1 - 1/\alpha)+ 1/\alpha} \frac{1}
{1+(n-1)(1- \frac 1 \alpha)}.
\end{split}
\end{equation}

It follows from (\ref{time-integral}), (\ref{general-integral})
and \eqref{Eq:fupb} that
\begin{equation}\label{Eq:Lupb}
\begin{split}
\E\big[|L_b - L_a|^n\big]&\leq \bigg(\frac{b-a} b\bigg)^{1/\alpha}\, \frac{\Gamma
(1- \frac 1 \alpha)^{n}\,C(\alpha)^n\, n!}
{ \Gamma \big(1+ n(1- \frac 1 \alpha)\big)}\, (b-a)^{n(1-1/\alpha)}.
\end{split}
\end{equation}
Recalling $A_1 = \Gamma (1- \frac 1 \alpha)\,C(\alpha)$, we see that
(\ref{Eq:Lupb}) gives the upper bound in (\ref{Eq:Ltmoments}).

Similarly, we use (\ref{time-integral}), (\ref{general-integral})
and \eqref{Eq:flowb} to derive
\begin{equation}\label{Eq:Llowb}
\E\big[|L_b - L_a|^n\big]\geq  \bigg(\frac{b-a} b\bigg)^{1/\alpha}
\frac{\Gamma (1- \frac 1 \alpha)^{n-1}\, C(\alpha)^n\, n!}{\Gamma
(1+ \frac 1 \alpha + n(1- \frac 1 \alpha))}\, (b-a)^{n(1-1/\alpha)},
\end{equation}
which yields the lower bound in (\ref{Eq:Ltmoments}). This proves Lemma
\ref{Lem:LevyLT}.
\end{proof}

Now we consider the moments of the increment $Z^H(b)-Z^H(a)$. As we
will see in Remark \ref{Re:KM}, the following lemma is sufficient for
proving (\ref{Eq:LD}) in Theorem \ref{probability-thm}.

\begin{lemma}\label{moment}
Let $W^H = \{W^H(t), t \in \rr\}$ be a fractional Brownian motion of index
$H$ in $\rr$, and $L_t$ be the local time at zero of a strictly stable
process $X = \{X_t, t \ge 0\}$ of
index $1<\alpha\leq 2$ independent of $W^H$. Then for all $0< a\leq b<\infty$
and all positive integers $n$,
\begin{equation}\label{a-positive}
\begin{split}
 C_1(n)\, (b-a)^{n(1-1/\alpha)} \leq
\E\left( \big|W^H(L_b)-W^H(L_a)\big|^{n/H} \right) \le
C_2(n)\, (b-a)^{n(1-1/\alpha)},
\end{split}
\end{equation}
where
\begin{equation}\label{Eq:C-Hn}
C_{1}(n)= \frac{1} {\sqrt{\pi} \Gamma\big(1 - \frac 1 \alpha\big)}\,
\bigg(\frac{b-a}{b}\bigg)^{1/\alpha}\,
\big( 2^{1/(2H)}\,A_1\big)^{ n}\,
\frac{n! \Gamma\big(\frac n{2H} + \frac 1 2\big)}
{\Gamma\big(1+ \frac 1 \alpha + n(1-\frac 1 \alpha)\big)}
\end{equation}
and
\begin{equation}\label{Eq:C-Hn2}
C_{2}(n)= \frac{1} {\sqrt{\pi}}\,\bigg(\frac{b-a}{b}\bigg)^{1/\alpha}\,
\big( 2^{1/(2H)}\,A_1\big)^{n}\,
\frac{n! \Gamma\big(\frac n{2H} + \frac 1 2\big)}
{\Gamma\big(1+ n(1-\frac 1 \alpha)\big)}.
\end{equation}
In the above $A_1 >0$ is the constant defined by (\ref{Eq:A1}).
Moreover, when $a=0$ we have the equality
\begin{equation}\label{a-zero}
\E\left( \big|W^H(L_b)\big|^{n/H} \right) =
\frac1 {\sqrt{\pi}}\,\big( 2^{1/(2H)}\,A_1\big)^{ n}\,
\frac{n! \Gamma\big(\frac n{2H} + \frac 1 2\big)}
{\Gamma\big(1+n(1-\frac 1 \alpha)\big)} \,
 b^{n(1-1/\alpha)}.
\end{equation}
\end{lemma}

\begin{proof}\
Since $L_t$ is a non-decreasing process, $W^H$ is $H$-self-similar
with stationary increments, and these two processes are independent,
we use a conditioning argument to derive
\begin{equation}\label{Eq:ZH-mom}
\E\left( \big|W^H(L_b)-W^H(L_a)\big|^{n/H}
\right)= \E\left(|W^H(1)|^{n/H}\right)\, \E\big(|L_b-L_a|^{n}\big).
\end{equation}
On the other hand, since $W^H(1)$ has a standard normal density, one can
use a change of variables to verify that
\begin{equation}\label{Eq:tauH}
\E\left(|W^H(1)|^{n/H}\right)=\frac{1} {\sqrt{\pi}}\, 2^{n/(2H)}\,
\Gamma \bigg( \frac n {2H} + \frac 1 2 \bigg).
\end{equation}
Combining (\ref{Eq:ZH-mom}), (\ref{Eq:tauH}) with \eqref{Eq:Ltmoments}
and \eqref{Eq:a-zero} proves (\ref{a-positive}) and (\ref{a-zero}).
\end{proof}

\section{Analytic results: exponential integrability}

In this section we study the exponential integrability of the random
variable $Z^H(t)$ and some analytic properties of its logarithmic
moment generating function. Our main results of this section are
Theorems \ref{Thm:Cumulant} and \ref{moment-generating-lim},  which
are the main ingredients for proving Theorems \ref{Thm:Zt-LDP} and
\ref{probability-thm}.

\begin{theorem}\label{Thm:Cumulant}
Let $Z^H= \{Z^H(t), t \ge 0\}$ be an $\a$-stable local time $H$-fractional
Brownian motion with values in $\rr$ and $2H < \alpha$. Then for every
$\theta \in \R$,
\begin{equation}\label{Eq:cumu}
\lim_{t \to \infty} t^{- \frac{2H(\alpha - 1)} {\alpha - 2H}}\,
\log \E\exp\big(\theta Z^H(t)\big) = \Lambda_1(\theta),
\end{equation}
where $\Lambda_1$ is the function on $\R$ defined by
$\Lambda_1(\theta) =B_1\, \theta^{\frac{2 \alpha} {\alpha - 2H}}$ for all
$\theta \in \R$. Recall from (\ref{Eq:ConstB}) and (\ref{Eq:A1}) that the constants
$B_1 = B_1(H, \alpha, \chi, \nu)$ and  $A_1$ are defined as
\begin{equation}\label{Eq:ConstB0}
B_1= \frac{\alpha - 2H} {2\alpha} \left(
\frac{H\, A_1^\alpha} {(1 - \frac 1 \alpha)^{\alpha -1}}\right)^{\frac{2H} {\alpha - 2H}}
\end{equation}
and
\begin{equation}\label{Eq:A1b}
A_1 = \frac{\Gamma(1-\frac 1 \alpha)\Gamma(\frac 1 \alpha)\chi^{1/\alpha}
\cos (\frac 1 \alpha \arctan
(\nu \tan (\frac{\pi \alpha}{2})))}{\pi \alpha [1+(\nu\tan
(\frac{\pi \alpha}{2}))^2]^{1/(2\alpha)}},
\end{equation}
where $\nu \in [-1,1]$
and $\chi>0$ are the constants defined in (\ref{Eq:Chf}).
\end{theorem}

Note that for $2H < \alpha$ the above function $\Lambda_1(\cdot)$ is even,
convex and differentiable on $\R$, and the function $\Lambda^*_1(\cdot)$
defined by (\ref{Eq:Lambdastar}) is the Fenchel-Legendre transform of
$\Lambda_1$, that is, $\Lambda^*_1(x) = \sup_{\theta \in \R}
\big( \theta x - \Lambda_1(\theta)\big)$ for all $x \in \R$.

The proof of Theorem \ref{Thm:Cumulant} relies on explicit calculation
of the moments of $Z^H(t)$ and the following theorem in Valiron
\cite[page 44]{valiron}.

\begin{lemma}\label{Lem:Valiron}
Let $f(z) = \sum_{p=0}^\infty c_p z^p$ be an entire function such that
$c_p \ne 0$ for infinitely many $p$'s. For any $r > 0$, let
$M(r)=\sup_{|z|=r}|f(z)|$.  Then a necessary and
sufficient condition for
\begin{equation}\label{Eq:ValironCon}
\lim_{r\to \infty}\frac{\log M(r)}{r^{\rho}}=B
\end{equation}
is that, for all values of $\varepsilon$ and all sufficiently large
integers $p$, we have
\begin{equation}\label{Eq:ValironCon1}
\frac{1}{\rho e}\, p\, c_{p}^{\rho/p}\leq B+\varepsilon ,
\end{equation}
and  there exists a sequence of integers $p_{n}$, such that
\begin{equation}\label{Eq:ValironCon2}
\lim_{n\to\infty}\frac{p_{n+1}} {p_{n}}=1,
\end{equation}
for which
\begin{equation}\label{Eq:ValironCon3}
\lim_{n\to\infty} \frac{1}{\rho e}\,p_{n}\, c_{p_{n}}^{\rho/p_{n}} =B.
\end{equation}
\end{lemma}

\begin{proof}[\it Proof of Theorem \ref{Thm:Cumulant}.]\ Similar to the proof
of Lemma \ref{moment}, we apply a conditioning
argument and the formula for the moment generating function of a Gaussian random
variable to derive that for all $\theta\in \R$,
\begin{equation}\label{Eq:cumu1}
\begin{split}
\E\exp\big(\theta Z^H(t)\big) &= \E
\exp\left( \frac {\theta^2} 2\, t^{2H(1- 1/\alpha)} L_1^{2H}\right).
\end{split}
\end{equation}
In order to prove (\ref{Eq:cumu}), we show that, for $2H < \alpha$, the function
$f(z) = \E\exp\big(z L_1^{2H}\big)$ is an entire function and the coefficients of
its Taylor expansion verify the conditions of Lemma \ref{Lem:Valiron}.

Let us first consider the Taylor series
\begin{equation}\label{Eq:cumu2}
M_1(r)= \sum_{n=0}^\infty \frac{\E(L_1^{2H n})} {n!}\, r^n.
\end{equation}
We will make use of the following consequence of Jensen's inequality:
For any constant $\gamma \ge 1$ and nonnegative random variable $\Delta$,
\begin{equation}\label{Eq:Jensen}
\Big( \E\big(\Delta^{\lfloor \gamma \rfloor}\big)\Big)^{\gamma/\lfloor \gamma \rfloor}
\le \E\big(\Delta^{\gamma}\big)\le \Big(\E\big(\Delta^{\lfloor \gamma \rfloor
+ 1}\big)\Big)^{\gamma/(\lfloor \gamma \rfloor + 1)}.
\end{equation}
Here and in the sequel, $\lfloor \gamma \rfloor$ denotes the largest integer
$\le \gamma$.

It follows from (\ref{Eq:Jensen}) with $\Delta = L_1$ and $\gamma = 2Hn$,
(\ref{Eq:a-zero}) in Lemma \ref{Lem:LevyLT} and Stirling's formula that
\begin{equation}\label{Eq:cumu3}
\frac{\E(L_1^{2H n})} {n!}\ \asymp \bigg(A_1\, \frac{(2H)^{1/\alpha}}
{(1 -\frac 1 \alpha)^{1- 1/\alpha}}\bigg)^{2Hn}\,
e^{n\big( 1- \frac{2 H} {\alpha}\big)}\,
n^{-n\big(1- \frac{2 H} {\alpha}\big)},
\end{equation}
where $A_1$ is the constant in (\ref{Eq:A1}). In the above, $x_n \asymp y_n$
means that, for all $n$ large enough, $x_n/y_n$ is bounded from below and
above by constant multiple of $n^{-\eta}$. Here $\eta$ is a
constant depending on $H$ and $\alpha$ only. The omitted factors have no
influence on the limit in (\ref{Eq:cumu4}) below.

By (\ref{Eq:cumu3}), we see that the Taylor series in (\ref{Eq:cumu2}) represents
an analytic function on $\R$ if and only if $2H < \alpha$. In the latter case,
we choose $\rho_1 = \frac{\alpha} {\alpha - 2H}$ and derive
\begin{equation}\label{Eq:cumu4}
\lim_{n \to \infty} \frac 1 {\rho_1 e} \, n\, \bigg(\frac{\E(L_1^{2H n})}
{n!}\bigg)^{\rho_1/n}
= \frac1 {\rho_1} \bigg(A_1\, \frac{(2H)^{1/\alpha}}
{(1 -\frac 1 \alpha)^{1- 1/\alpha}}\bigg)^{2H \rho_1}.
\end{equation}
Hence, Lemma \ref{Lem:Valiron} implies that
\begin{equation}\label{Eq:cumu5}
\lim_{r \to \infty} \frac{\log M_1(r)} {r^{\rho_1}} = \frac1 {\rho_1}
\bigg(A_1\, \frac{(2H)^{1/\alpha}}
{(1 -\frac 1 \alpha)^{1- 1/\alpha}}\bigg)^{2H \rho_1}.
\end{equation}

It follows from (\ref{Eq:cumu1}) that $\E\exp\big(\theta Z^H(t)\big)
= M_1\big(\frac {\theta^2} 2\, t^{2H(1- 1/\alpha)}\big)$. Hence (\ref{Eq:cumu})
follows from (\ref{Eq:cumu5}) and a simple change of variables.
\end{proof}

The proof of Theorem \ref{Thm:Cumulant} shows that, if $2H > \alpha$, then
$\E\big(e^{\theta Z^H(t)}\big) = \infty$ for all $\theta >0$. Hence we cannot
prove a large deviation principle for $Z^H(t)$ by applying the G\"artner-Ellis
Theorem. However, for studying the tail probability of
$Z^H(b)-Z^H(a)$, it is sufficient to consider the  exponential integrability
of $|Z^H(b)-Z^H(a)|^\beta$ for appropriately chosen $\beta> 0$.

\begin{prop}\label{Prop:Ana}
Let $0 \le a < b < \infty$ be given constants. For any $\beta > 0$ and $t \in \R$,
let $g_\beta (t) = \E\big(e^{t |W^H(L_b)-W^H(L_a)|^{\beta}}\big)$.
The following statements hold:
\begin{itemize}
\item[(i)]\ If $0 < \beta < \frac{2 \alpha} {2H + \alpha}$, then the function
$g_\beta(t)$ is analytic on $\R$.
\item[(ii)]\ If $\beta = \frac{2 \alpha} {2H + \alpha}$, then $g_\beta(t)$ is analytic
in $(-\infty, \delta_0)$ for some $\delta_0 > 0$.
\item[(iii)]\ If $\beta > \frac{2 \alpha} {2H + \alpha}$, then $ g_\beta(t) = \infty$
for all $t > 0$.
\end{itemize}
\end{prop}

\begin{proof}\ Let us consider the Taylor series
\[
\sum_{n=0}^\infty \frac{\E\big(|W^H(L_b)-W^H(L_a)|^{\beta n}\big)} {n!}\, t^n.
\]
As in the proof of Lemma \ref{moment}, we have
\begin{equation}\label{Eq:ZH-mom2}
\begin{split}
\E\big(|W^H(L_b)-W^H(L_a)|^{\beta n}\big) &= \E\big(|W^H(1)|^{\beta n}\big)\,
\E\big(L_b - L_a|^{\beta H n }\big)\\
&= \frac{1} {\sqrt{\pi}}\, 2^{(\beta n)/2}\,
\Gamma \bigg( \frac {\beta n} {2} + \frac 1 2 \bigg)\,
\E\big(|L_b - L_a|^{\beta H n}\big).
\end{split}
\end{equation}

Applying (\ref{Eq:Jensen}) to $\Delta = |L_b - L_a|$ and $\gamma = \beta H n$
and using (\ref{Eq:ZH-mom2}), the moment estimates in Lemma \ref{Lem:LevyLT}
and Stirling's formula, we derive
\begin{equation}\label{Eq:Jensen2}
\frac{\E\big(|W^H(L_b)-W^H(L_a)|^{\beta n}\big)} {n!} \asymp A_2^n\,
e^{n\big( 1-\frac \beta 2 - \frac{\beta H} {\alpha}\big)}\,
n^{-n\big(1-\frac \beta 2 - \frac{\beta H} {\alpha}\big)},
\end{equation}
where $A_2$ is a constant which can be expressed explicitly in terms
of $\alpha, \beta, H$, $A_1$ and $b-a$. It can be verified that
(\ref{Eq:Jensen2}) implies the conclusions in Proposition \ref{Prop:Ana}.
\end{proof}

It follows from Proposition \ref{Prop:Ana} that, for $\beta > 0$, $g_\beta (z)$
($z \in \C$) is an entire function if and only if $\beta < \frac{2 \alpha}
{2H + \alpha}$. In this case, Theorem \ref{moment-generating-lim} further
proves that $g_\beta (z)$ is of \emph{very regular growth} in the sense of
Valiron (\cite{valiron}).

\begin{theorem}\label{moment-generating-lim}
Let $W^H$ be a fractional Brownian motion in $\rr$, and let $L_t$ be the local time
at zero of a strictly stable process $X_t$ of index $1<\alpha\leq 2$
independent of $W^H$. Then for all $0 < \beta < \frac{2 \alpha}
{2H + \alpha}$ and $0\leq a\leq b<\infty$,
\begin{equation}\label{Eq:Entirelimit}
\lim_{t\to\infty}\frac{\log \E\big[\exp\big(t|W^H(L_b)-W^H(L_a)|^\beta\big)\big]}
{t^{\rho}}=B_3,
\end{equation}
where $\rho = \frac{2 \alpha} { 2 \alpha- \alpha \beta - 2H \beta}$  and
$B_3 = B_3(H, \alpha, \nu, \chi, \beta)$ is the constant given by
\begin{equation}\label{Eq:conB}
B_3 =\frac 1 {\rho} \,A_1^{\beta H \rho} \, (b-a)^{\beta H
\rho (1 - 1/\alpha)}\Bigg(\frac{\beta^{\beta/(2H)}
(\beta H)^{\beta/\alpha}}
{\big(1 - \frac 1 \alpha\big)^{\beta  (1 - \frac 1 \alpha)}}\Bigg)^{H\rho},
\end{equation}
where $A_1$ is the constant given in (\ref{Eq:A1}).
\end{theorem}

Note that, when $H=1/2$, we can take $\beta = 1$. In this case, $B_3$ is reduced to
$$
B_3 =\left[\frac{\Gamma (1- \frac 1 \alpha)
\Gamma (\frac 1 \alpha)\chi ^{1/\alpha}\cos \big(\frac 1 \alpha
 \arctan (\nu \tan (\frac{\pi \alpha}{2})) \big)}{2\pi \alpha
 [1+(\nu\tan (\frac{\pi \alpha}{2}))^2]^{1/(2\alpha)}}
 \right]^{\frac{\alpha}{\alpha-1}}.
 $$

\begin{proof}[\it Proof of Theorem \ref{moment-generating-lim}]\
We start with the following elementary fact: Let
$H \in (0, 1)$ be a constant. Then for all $x \ge 0$,
\begin{equation}\label{Eq:AuxInq}
1+ \sum_{n=1}^{\infty} \frac{x^{\lfloor n/H \rfloor}}
{\lfloor \frac n H \rfloor !} \leq e^{x}
\leq \frac e H \, \left( 1+ \sum_{n=1}^{\infty}
\frac{x^{\lfloor n/H \rfloor}}
{\big(\lfloor \frac {n-1} H \rfloor +1 \big)!} \right).
\end{equation}
In order to verify (\ref{Eq:AuxInq}), first note that the first inequality holds
for all $x \ge 0$ because $\{\lfloor \frac n H \rfloor, n \ge 0\}$ is a subsequence
of $\N$, and the second inequality holds for all $0 \le x \le 1$.
Hence it only remains to show the second inequality holds for all $x > 1$.
This can be verified by grouping the terms in the expansion
$e^x = 1 +\sum_{k=1}^\infty \frac{x^k} {k!}$ in the blocks
$k \in \{\lfloor \frac {n-1} H \rfloor +1, \cdots, \lfloor \frac{n} H \rfloor \}$,
and noting that the number of integers in each block is at most $1 + \frac 1 H$,
which is smaller than $\frac e H$.

Now let $\beta \in (0,\frac{2 \alpha} {2H + \alpha})$ be a constant and let
$t \ge 0$. By taking $x = t |W^H(L_b)-W^H(L_a)|^\beta$ in (\ref{Eq:AuxInq}),
we have
\begin{equation}\label{Eq:twoF}
\begin{split}
1+ &\sum_{n=1}^{\infty} \frac{t^{\lfloor n/H \rfloor} \E\big(
|W^H(L_b)-W^H(L_a)|^{\beta \lfloor n/H \rfloor}\big)}
{\lfloor \frac n H \rfloor !}
\leq g_\beta(t)\\
&\qquad \qquad \qquad \leq \frac e H \, \left( 1+
\sum_{n=1}^{\infty} \frac{t^{\lfloor n/H \rfloor}
\E \big(|W^H(L_b)-W^H(L_a)|^{\beta \lfloor n/H \rfloor}\big)}
{\big(\lfloor \frac {n-1} H \rfloor +1 \big)!} \right).
\end{split}
\end{equation}
Denote the first and last terms in (\ref{Eq:twoF}) by $f_1(t)$ and $f_2(t)$,
respectively. In order to prove (\ref{Eq:Entirelimit}), it suffices to show
that for $i=1, 2$,
\begin{equation}\label{Eq:twoF2}
\lim_{t \to \infty} \frac{\log f_i(t)} {t^{\rho}} = B_3,
\end{equation}
where $\rho = \frac{2 \alpha} { 2 \alpha- \alpha \beta - 2H \beta}$ and
$B_3 = B_3(H, \alpha, \nu, \chi, \beta)$ is given by (\ref{Eq:conB}).

This can be done by showing the coefficients of $f_i$ satisfy the conditions of
Lemma \ref{Lem:Valiron}. Moreover, since the proofs for $i = 1, 2$ are
almost the same, we only prove (\ref{Eq:twoF2}) for $i= 1$.

Note that the coefficients $c_p$ ($p = 0, 1, \cdots$) of $f_1(t)$ are given by
$$
c_p = \frac{\E\big(
|W^H(L_b)-W^H(L_a)|^{\beta \lfloor n/H \rfloor}\big)} {\lfloor \frac n H \rfloor !}
$$
if $p = \lfloor \frac n H\rfloor$ and $c_p = 0$ otherwise. By Lemma \ref{Lem:Valiron},
it suffices to show
\begin{equation}\label{Eq:C-lim}
\lim_{n \to \infty} \frac 1 {\rho e}\, \lfloor \frac n H \rfloor\,
\left(\frac{ \E\big(|W^H(L_b)-W^H(L_a)|^{\beta \lfloor n/H \rfloor}\big)}
{\lfloor \frac n H \rfloor !}\right)^{\rho/\lfloor \frac n H\rfloor}
= B_3.
\end{equation}

As in the proof of Lemma \ref{moment}, we have
\begin{equation}\label{Eq:ZH-mom3}
\begin{split}
&\E\big(|W^H(L_b)-W^H(L_a)|^{\beta \lfloor n/H \rfloor} \big) \\
&= \frac{1} {\sqrt{\pi}}\, 2^{\beta \lfloor n/H \rfloor/2}\,
\Gamma \bigg( \frac {\beta \lfloor n/H \rfloor} {2} + \frac 1 2 \bigg)\,
\E\big(|L_b - L_a|^{\beta H \lfloor n/H \rfloor}\big).
\end{split}
\end{equation}
By (\ref{Eq:Jensen}), the moment estimates in Lemma \ref{Lem:LevyLT} and
Stirling's formula, we can verify that
\begin{equation}\label{Eq:Jensen3}
\frac{\E\big(|W^H(L_b)-W^H(L_a)|^{\beta \lfloor n/H \rfloor}\big)}
{\lfloor \frac n H \rfloor !} \asymp A_3^{ \lfloor n/H \rfloor}\,
e^{n\big( \frac 1 H -\frac \beta {2H} - \frac{\beta } {\alpha}\big)}\,
n^{-n\big(\frac 1 H -\frac \beta {2H} - \frac{\beta} {\alpha}\big)},
\end{equation}
where $A_3$ is a constant defined by
\begin{equation}\label{Eq:A3}
A_3 =  A_1^{\beta H}\, (b-a)^{\beta H(1 - \frac 1 \alpha)}
\left(\frac{\beta^{\frac \beta {2H} + \frac \beta \alpha} \,
H^{\frac 1 H - \frac \beta {2H}}}
{\big(1 - \frac 1 \alpha\big)^{\beta  (1 - \frac 1 \alpha)}}\right)^H.
\end{equation}
Since $\big(\frac 1 H - \frac \beta {2H} - \frac{\beta } {\alpha}\big)H \rho = 1$,
we see that (\ref{Eq:Jensen3}) implies
\begin{equation}\label{Eq:C-lim2}
\lim_{n \to \infty} \frac 1 {\rho e}\, \lfloor \frac n H \rfloor\,
\left(\frac{ \E\big(|W^H(L_b)-W^H(L_a)|^{\beta \lfloor n/H \rfloor}\big)}
{\lfloor \frac n H \rfloor !}\right)^{\rho/\lfloor \frac n H\rfloor}
= \frac{A_3^{\rho}} {\rho H} = B_3,
\end{equation}
where $B_3$ is given by (\ref{Eq:conB}).
This proves (\ref{Eq:C-lim}) and hence Theorem \ref{moment-generating-lim}.
\end{proof}

\section{Large deviations results: proofs of Theorems \ref{Thm:Zt-LDP}
and \ref{probability-thm}}

In this section, we first prove Theorems \ref{Thm:Zt-LDP} and
\ref{probability-thm}. Then we apply a maximal inequality due
to Moricz, et al. \cite{MSS82} to
derive upper bounds for the tail probabilities of the maxima
$\max_{t \in [0, 1]}Z^H(t)$ and $\max_{t \in [a, b]}|Z^H(t) - Z^H(a)|$.

\begin{proof}[\it Proof of Theorem \ref{Thm:Zt-LDP}]\ Note that the function
$\Lambda_1 (\theta)= B_1\, \theta^{\frac{2 \alpha} {\alpha - 2H}}$ in
Theorem \ref{Thm:Cumulant} is essentially smooth and continuous
on $\R$. It follows from the G\"artner-Ellis Theorem
(cf. \cite[Theorem 2.3.6]{DemboZ98}) that the pair $\big(t^{- \frac{2H(\alpha - 1)}
{\alpha - 2H}}\,Z^H(t),
t^{\frac{2H(\alpha - 1)} {\alpha - 2H}}\big)$ satisfies a large deviation principle
with the good rate function
$$
\Lambda^*_1(x) = \sup_{\theta \in \R}
\big( \theta x - \Lambda_1(\theta)\big),$$
which is the Fenchel-Legendre transform of $\Lambda_1$.
It is elementary to verify that $\Lambda^*_1(x)$
coincides with (\ref{Eq:Lambdastar}). This proves Theorem \ref{Thm:Zt-LDP}.
\end{proof}

\begin{proof}[\it Proof of Theorem \ref{probability-thm}]\
Let $0 < \beta < \frac{2 \alpha} {2H + \alpha}$ be fixed. It follows from
(\ref{Eq:Entirelimit}) in Theorem \ref{moment-generating-lim} and
Davies' Theorem 1 in \cite{davies} that
\begin{equation}\label{Eq:Davies1}
\lim_{u \to \infty} \frac{\log
\P\big\{|W^H(L_b)-W^H(L_a)|^\beta \ge u\big\}} {u^{\rho/(\rho - 1)}}
= - \big(1 - \rho^{-1}\big) \big(\rho B_3\big)^{-1/(\rho-1)}.
\end{equation}
Here $\rho = \frac{2 \alpha} { 2 \alpha- \alpha \beta - 2H \beta}$.
Letting $x = u^{1/\beta}$ and simplifying the right-hand side of (\ref{Eq:Davies1}),
we obtain
\begin{equation}\label{Eq:Davies2}
\begin{split}
&\lim_{x \to \infty} \frac{ \log \P\big\{|W^H(L_b)-W^H(L_a)| \ge x\big\} }
{x^{\frac{2 \alpha} {\alpha + 2H}}} \\
& = - \frac{\alpha + 2H} {2 \alpha}\, \left(\frac{H\, {A}_1^\alpha}
{\big(1 - \frac 1 \alpha\big)^{\alpha -1 }}\,
 \right)^{- \frac{2 H} {\alpha + 2H}}\, \big(b-a\big)^{- \frac{2H(\alpha - 1)}
{\alpha + 2H}}.
\end{split}
\end{equation}
This finishes the proof of Theorem \ref{probability-thm}.
\end{proof}

\begin{remark}\label{Re:KM}
Theorem \ref{probability-thm} can also be proved by using Lemma 2.3 in K\"{o}nig
and M\"{o}rters \cite{kon-mor} (note that their assumption $p \in \N$ can be
replaced by $p>0$) and the moment results in Lemma \ref{moment}.
The proof of Lemma 2.3 in \cite{kon-mor} is based on a change-of-measure technique
in large deviations. We remark that Lemma 2.3 in \cite{kon-mor} is equivalent
to Corollary 2 in Davies \cite{davies}, hence it can also be proved by using
an analytic method.
\end{remark}

Similar to the proof of Theorems \ref{Thm:Cumulant} and \ref{Thm:Zt-LDP},
we obtain the following theorem for the local time $L_t$ at zero of $X$.
They are in complement to the results of Hawkes \cite{hawkes} and Lacey \cite{lacey}
on the tail asymptotics for the local time $L_t$ and the maximum local time
$\max_{x \in \R} L(x, t)$, respectively.

\begin{theorem}\label{moment-generating-lim-local}
Let $L = \{L_t, t \ge 0\}$ be the local time at zero of a real-valued strictly
stable L\'evy process $X= \{X_t, t \ge 0\}$ of index $1<\alpha\leq 2$. Then the
following two statements hold:

{\rm (i)}\ For all $0\leq a\leq b<\infty$,
\begin{equation}\label{Eq:LTgen}
\lim_{t\to\infty}\frac{\log
\E[\exp( t\,(L_b-L_a))]}{t^{\frac{\alpha}{\alpha-1}}} = (b-a)\,C(\alpha,
\nu,\chi),
\end{equation}
where $C(\alpha, \nu,\chi)$ is the constant defined by
\[
C(\alpha, \nu,\chi)=\left[\frac{\Gamma \big(1- \frac 1 \alpha\big)\Gamma
\big(\frac 1 \alpha\big)\chi ^{1/\alpha}\cos \big(\frac 1 \alpha
 \arctan (\nu \tan (\frac{\pi \alpha}{2}))\big)}{\pi \alpha
 \big[1+(\nu\tan (\frac{\pi \alpha}{2}))^2\big]^{1/2\alpha}}
 \right]^{\frac{\alpha}{\alpha-1}}.
\]

{\rm (ii)}\ The pair $\big(t^{-1/(\alpha-1)} (L_b- L_a),\, t^{\alpha/(\alpha-1)}\big)$
satisfies LDP with good rate function $\Lambda^*_2(x) = \frac{x^{\alpha}}
{\alpha}\, \left[\left(\frac{\alpha}{\alpha  -1}
\right)(b-a)\, C(\alpha, \nu,\chi)\right]^{-(\alpha-1)}$ if $x > 0$ and
$\Lambda^*_2(x) = \infty$ if $x \le 0$. That is,
for every Borel set $F\subseteq \rr$,
\begin{equation}\label{LT:LDPup}
\limsup_{t \to \infty} t^{-\frac{\alpha}{\alpha-1}}\, \log \P\Big\{
t^{- \frac{1} {\alpha - 1}} \, (L_b-L_a) \in F\Big\}
\le - \inf_{x \in \overline{F}}\Lambda^*_2(x)
\end{equation}
and
\begin{equation}\label{LT:LDPlow}
\liminf_{t \to \infty} t^{- \frac{\alpha}{\alpha-1}}  \, \log \P\Big\{
t^{- \frac{1} {\alpha - 1}} \, (L_b-L_a) \in F\Big\}
\ge - \inf_{x \in {F}^{\circ}}\Lambda^*_2(x).
\end{equation}
\end{theorem}

\begin{proof}\ Similar to the proof of Theorem \ref{Thm:Cumulant},
Eq. (\ref{Eq:LTgen}) follows from Lemma \ref{Lem:Valiron} and the moment estimates
in Lemma \ref{Lem:LevyLT}.

It follows from (i) that for all $\theta >0$,
\begin{equation}\label{Eq:LTgen2}
\lim_{t\to\infty}\frac{\log
\E[\exp(\theta\, t\,(L_b-L_a))]}{t^{\frac{\alpha}{\alpha-1}}} = (b-a)\,C(\alpha,
\nu,\chi)\, \theta^{\alpha/(\alpha -1)}.
\end{equation}
Denote
\begin{equation}\label{Eq:Lambda2}
\Lambda_2(\theta) =\left\{\begin{array}{ll}
(b-a)\,C(\alpha,\nu,\chi)\, \theta^{\alpha/(\alpha -1)}
 \qquad &\hbox{ if }\ \theta > 0,\\
0 &\hbox{ if }\ \theta \le 0.\end{array}\right.
\end{equation}
Then, in the terminology of \cite{DemboZ98}, $\Lambda_2$ is an essentially
smooth, continuous function on $\R$ and its
Fenchel-Legendre transform is given by
\begin{equation}\label{Eq:Lambda2*}
\Lambda^*_2(x)  =\left\{\begin{array}{ll}
\frac{x^{\alpha}} {\alpha}\, \left[\left(\frac{\alpha}{\alpha  -1}
\right)(b-a)\, C(\alpha, \nu,\chi)\right]^{-(\alpha-1)}
 \qquad &\hbox{ if }\ x > 0,\\
\infty &\hbox{ if }\ x \le 0.\end{array}\right.
\end{equation}
Therefore, as in the proof of Theorem \ref{Thm:Zt-LDP},
Part (ii) follows from (\ref{Eq:LTgen2}) and the G\"artner-Ellis Theorem.
\end{proof}

\begin{remark}
Let $F=[1, \infty)$ in (\ref{LT:LDPup}) and (\ref{LT:LDPlow}), we obtain the tail
probability
\begin{eqnarray}\label{Eq:LT:tail}
\lim_{x \to\infty} \frac{\log \P \left\{ L_b-L_a>x
\right\}}{x^{\alpha}}&=& - \frac{1}{\alpha} \left[\left(\frac{\alpha}{\alpha  -1}
\right)(b-a)\, C(\alpha, \nu,\chi)\right]^{-(\alpha-1)}.
\end{eqnarray}
In case $a=0$, the limit in (\ref{Eq:LT:tail}) is weaker than the best
known. By using the connection between $L_t$ and a stable subordinator of
index $1 - 1/\alpha$,  one can derive more precise result (i.e., without the logarithm)
on limiting behavior of $\P \big\{ L_b >x \big\}$ as $x \to \infty$. See Hawkes
\cite{hawkes}.
\end{remark}

Theorem \ref{probability-thm} is concerned with asymptotic behavior of the tail
probability $\P \big\{ |Z^H(b)$ $ - Z^H(a)| >x \big\}$ as $x \to \infty$.
In many applications, however, it is useful to have sharp bounds on
$\P \big\{ |Z^H(b) - Z^H(a)| >x \big\}$ and $\P \big\{ \max_{a \le t \le b}
|Z^H(t) - Z^H(a)| >x \big\}$ that hold for \emph{all} $x > 0$. In the rest of this
section, we consider these questions and in the next section we use them to derive
upper bounds for the local and uniform moduli of continuity for $Z^H$.

\begin{lemma}\label{for-all-u}
There exists a finite constant $A_4>0$, depending on $H, \alpha, \nu$ and $\chi$ only,
such that for all  $0\le a<b < \infty$ and all $x>0$,
\begin{equation}\label{Eq:tailUPB}
\P \big\{|Z^H(b)-Z^H(a)|>  x\big\} \leq  \exp\left(- A_4\,
\frac{x^{2\alpha/(\alpha+2H)}} {(b-a)^{2H(\alpha-1)/(\alpha + 2H)}}\right).
\end{equation}
\end{lemma}

\begin{proof}\ We consider the random variable
$$
\Lambda =\frac{|Z^H(b)-Z^H(a)|}{(b-a)^{H(\alpha-1)/\alpha}}.
$$
As in the proof of Lemma \ref{moment}, we apply a conditioning argument and
Lemma \ref{Lem:LevyLT} to show that for all integers $n \ge 1$,
\begin{equation}\label{Eq:Lamom}
\E(\Lambda^n) \le  A_5^n\, n^{n(\alpha + 2H)/(2\alpha)},
\end{equation}
where $A_5 > 0$ is a constant depending on $H, \alpha, \nu$ and $\chi$ only.

For any constant $A_6> 0$, the Markov inequality and (\ref{Eq:Lamom}) imply that
for all $u> 0$,
\begin{equation} \label{Eq:Lamom2}
\begin{split}
\P\big\{\Lambda > A_6 u \big\} &\leq
\frac{A_5^n\, n^{n(\alpha + 2H)/(2\alpha)}} {A_6^n\,u^n} \\
&= \bigg(\frac{A_5} {A_6}\bigg)^n\,
\bigg(\frac{n^{(\alpha + 2H)/(2\alpha)}} u\bigg)^n.
\end{split}
\end{equation}
By taking the constant $A_6 \ge e A_5$ and $n = \lfloor u^{2\alpha/(\alpha + 2H)}\rfloor$,
we obtain
\begin{equation} \label{Eq:Lamom3}
\begin{split}
\P\big\{|Z^H(b)-Z^H(a)| > A_6\,(b-a)^{H(\alpha-1)/\alpha}\, u \big\} \leq
\exp \left(- u^{2\alpha/(\alpha+2H)} \right).
\end{split}
\end{equation}
It is clear that (\ref{Eq:tailUPB}) follows from (\ref{Eq:Lamom3}) by
letting $x = A_6\,(b-a)^{H(\alpha-1)/\alpha} u$.
\end{proof}

Next we apply Lemma \ref{for-all-u} and a result of Moricz, et al. \cite{MSS82}
to prove the following theorem.
\begin{theorem}\label{Thm:maxtail}
There exist positive and finite constants $A_7$ and $A_8$, depending on $H,
\alpha, \nu$  and $\chi$ only, such that for all  $0\le a< b < \infty$ and
all $x>0$,
\begin{equation}\label{Eq:maxtailUPB}
\P \bigg\{\max_{a \le t \le b}|Z^H(t)-Z^H(a)|> x\bigg\}
\leq A_7\, \exp\left(- A_8\, \frac{x^{2\alpha/(\alpha+2H)}}
{(b-a)^{2H(\alpha-1)/(\alpha + 2H)}}\right).
\end{equation}
\end{theorem}

\begin{proof}\ For any integer $n \ge 2$, we divide the interval $[a,\, b]$ into $n$
subintervals of  length $(b-a)/n$. Let $t_{n, i}= a + \frac{i(b-a)}n$ ($ i\in
\{0, 1, \ldots,n\}$) be the end-points of these subintervals. By the sample path
continuity of $Z^H$, it suffices to show that
\begin{equation}\label{Eq:maxtailUB1}
\P \bigg\{\max_{1\le i \le n}|Z^H(t_{n, i})-Z^H(a)|> x\bigg\}
\leq A_7\, \exp\left(- A_8\, \frac{x^{2\alpha/(\alpha+2H)}}
{(b-a)^{2H(\alpha-1)/(\alpha + 2H)}}\right)
\end{equation}
for all integers $n \ge 2$.

To this end, we define the random variables $\xi_i = Z^H(t_{n, i+1}) - Z^H(t_{n, i})$ for
$i \in \{0, 1, \ldots, n-1\}$. Then for all integers $0 \le j < k \le n$, we have
\begin{equation}\label{Eq:Incre}
Z^H(t_{n, k}) - Z^H(t_{n, j}) = \sum_{i=j}^{k-1} \xi_i : = S(j, k).
\end{equation}
Applying Lemma \ref{for-all-u}, we see that for all integers $j < k$,
\begin{equation}\label{Eq:maxtailUB2}
\P \big\{|S(j, k)| > x\big\}
\leq \exp\left(- A_4\, \bigg(\frac{n} {(b-a)(k-j)}\bigg)^{2H(\alpha-1)/(\alpha + 2H)}
\, x^{2\alpha/(\alpha+2H)}\right).
\end{equation}
Using the notation in \cite{MSS82}, we denote $\phi(x) = x^{2\alpha/(\alpha+2H)}$ and
$$
g(j, k) = A_4\, \bigg(\frac {(b-a)(k-j)} {n} \bigg)^{2H(\alpha-1)/(\alpha + 2H)}.
$$
For simplicity denote $r = 2H(\alpha-1)/(\alpha + 2H)$. Since $r \in (0,  1)$, the concavity
of the function $t \mapsto t^r$ implies that for all integers $1 \le i \le j < k\le n$,
\begin{equation}\label{Eq:Q-super}
 g(i, j) + g(j+1, k) \le 2^{1 - r }\, g(i, k).
\end{equation}
Hence the function $g$ satisfies the property of \emph{quasi-superadditivity} with
index $Q=2^{1 - r }$ in \cite{MSS82}. Moreover, the functions $\phi$ and $g$ satisfy
all the other conditions of Theorem 2.2 in Moricz, et al. \cite{MSS82}. Consequently,
the latter implies the existence of positive and finite constants $A_7$ and $A_8$
(depending on $H, \alpha, \nu$  and $\chi$ only) such
that (\ref{Eq:maxtailUB1}) holds. This proves Theorem \ref{Thm:maxtail}.
\end{proof}

\begin{remark}
Note that, when $H=1/2$, one can apply the reflection principle of Brownian motion
and conditioning to prove Theorem \ref{Thm:maxtail}. Our method is much more general.
\end{remark}

\section{Applications}

Applying the large deviation results in the previous section, we establish uniform and local
moduli of continuity for $Z^H$.

\begin{theorem}\label{main-thm}
Let $Z^H= \{Z^H(t), t \ge 0\}$ be an $\a$-stable local time $H$-fractional
Brownian motion with values in $\rr$. Then there exists a finite constant $A_9>0$
such that for all constants $0\le a<b <\infty$, we have
\begin{equation}\label{Eq:uniform}
\limsup_{h\downarrow 0} \sup_{a \leq t\leq b -h}\sup_{0\leq
s\leq h} \frac{\big|Z^H(t+s)-Z^H(t) \big|}{h^{H(\alpha-1)/\alpha}\big(
\log 1/h \big)^{(\alpha+2H)/(2\alpha)}}\leq A_9 \quad \ \
\mathrm{a.s.}
\end{equation}
\end{theorem}

\begin{proof}
For every $t \ge 0$ and $h > 0$, it follows from (\ref{Eq:Lamom3}) that
\begin{equation}\label{Eq:Lamom4}
\P\big\{|Z^H(t+h)-Z^H(h)| > A_6\,h^{H(\alpha-1)/\alpha}\, u \big\} \leq
\exp \left(- u^{2\alpha/(\alpha+2H)} \right).
\end{equation}
Hence $Z^H = \{Z^H(t), t \ge 0\}$ satisfies the conditions of
Lemmas 2.1 and 2.2 in \cite{csaki-csorgo} with $\sigma(h) =
h^{H(\alpha-1)/\alpha}$ and $\beta = \frac{2\alpha} {\alpha+2H}$.
Consequently (\ref{Eq:uniform}) follows
directly from Theorem 3.1 in \cite{csaki-csorgo}.
\end{proof}

Cs\'{a}ki,  F\"{o}ldes and R\'{e}v\'{e}sz \cite{CCFR2} obtained a
Strassen type law of the iterated logarithm (LIL) for $Z(t) = W(L_t)$
when $L_t$ is the local time at zero of a \emph{symmetric} stable L\'{e}vy
process (see Theorem 2.4 in \cite{CCFR2}). Part (i) of the following theorem
extends partially their result to $Z^H$ and Part (ii) describes the local
oscillation of $Z^H$ in the neighborhood of any fixed point.

\begin{theorem}\label{LIL}
Let $Z^H= \{Z^H(t), t \ge 0\}$ be an $\a$-stable local time $H$-fractional
Brownian motion with values in $\rr$. The following statements hold:

{\rm (i)}\ Almost surely,
\begin{equation}\label{Eq:LIL1}
\limsup_{t\to\infty}\frac{\max_{0\le s \le t} \big|Z^H(s)\big|}
{t^{H(\alpha-1)/\alpha}\big(\log\log t\big)^{(\alpha+2H)/(2\alpha)}}
\leq A_8^{-(\alpha +2H)/(2\alpha)}.
\end{equation}

{\rm (ii)}\ For every $t > 0$, almost surely
\begin{equation}\label{Eq:LIL2}
\limsup_{h \to 0}\frac{\max_{|s| \le h} \big|Z^H(t+s) - Z^H(t)\big|}
{h^{H(\alpha-1)/\alpha}\big(\log\log 1/h\big)^{(\alpha+2H)/(2\alpha)}}
\leq A_8^{-(\alpha +2H)/(2\alpha)}.
\end{equation}
In the above, $A_8$ is the constant in (\ref{Eq:maxtailUPB}).
\end{theorem}

\begin{proof}\ Since both (\ref{Eq:LIL1}) and (\ref{Eq:LIL2}) follow from
Theorem  \ref{Thm:maxtail} and a standard Borel-Cantelli argument, we only
prove (\ref{Eq:LIL1}).

Fix two arbitrary constants $\gamma > A_8^{-1}$ and $\rho > 1$. For every
integer $n \ge 1$, let
$T_n = \rho^n$ and consider the event
\[
E_n=\left\{\omega :\ \max_{0\leq t\leq T_n} \big|Z^H(s)\big|
>  T_n^{H(\alpha-1)/\alpha}\, U(T_n)\right\},
\]
where $U(t)=(\gamma\, \log \log t)^{(\alpha + 2H)/(2\alpha)}$.
It follows from Theorem \ref{Thm:maxtail} that
\begin{equation}\label{Eq:En}
\begin{split}
\P(E_n) &\le A_7\, \exp\left(- A_8\, \frac{\big(T_n^{H(\alpha-1)/\alpha}
\, U(T_n)\big)^{2\alpha/(\alpha+2H)}}
{T_n^{2H(\alpha-1)/(\alpha + 2H)}}\right)
 \le  A_{10}\, n^{- A_8 \gamma}.
\end{split}
\end{equation}
Since $A_8 \gamma> 1$, we have $\sum_{n=1}^{\infty}\P(E_n)<\infty$. The
Borel-Cantelli lemma implies that
\begin{equation}\label{upper-bound}
\limsup_{n\to\infty}\frac{\max_{0\leq s\leq T_n} \big|Z^H(s)\big|}
{T_n^{H(\alpha-1)/\alpha}\, U(T_n)} \leq 1\quad \ \hbox{ a.s.}
\end{equation}

Note that $T_{n+1}/T_n=\rho$ for every $n \ge 1$. We use the
monotonicity to derive that
for all $t\in [T_n,\, T_{n+1}]$,
\begin{equation}\label{Eq:Compare}
\frac{\max_{0\leq s\leq t} \big|Z^H(s)|} {t^{H(\alpha-1)/\alpha}
U(t)} \leq \rho^{H(\alpha-1)/\alpha}\, \frac{\max_{0\leq s\leq
T_{n+1}} \big|Z^H(s)\big|}{T_{n+1}^{H(\alpha-1)/\alpha}\ U(T_{n+1})}\,
\frac{U(T_{n+1})}{U(T_n)}.
\end{equation}
Equations (\ref{upper-bound}) and (\ref{Eq:Compare}) imply
\begin{equation}\label{Eq:LIL3}
\limsup_{t\to\infty }\frac{\max_{0\leq s\leq t} \big|Z^H(s)\big|}
{t^{H(\alpha-1)/\alpha}\big(\log\log t\big)^{(\alpha+2H)/(2\alpha)}}
\leq \rho^{H(\alpha-1)/\alpha}\,\gamma^{(\alpha+2H)/(2\alpha)}\quad \
\hbox{ a.s.}
\end{equation}
We obtain (\ref{Eq:LIL1}) from (\ref{Eq:LIL3}) by letting $\gamma \downarrow
A_8^{-1}$ and $\rho \downarrow 1$ along rational numbers.
\end{proof}

\begin{remark}
We believe that, up to a constant factor, both uniform and local moduli
of continuity of $Z^H(t)$ are sharp. However, we have not been able to
prove this due to the lack of information on the dependence structure
of the process $Z^H$.
\end{remark}

\end{document}